\documentclass[a4paper,fleqn]{cas-sc}

\usepackage[authoryear,longnamesfirst]{natbib}

\usepackage{graphicx}%
\usepackage{multirow}%
\usepackage{amsmath,amssymb,amsfonts}%
\usepackage{amsthm}%
\usepackage{mathrsfs}%
\usepackage[title]{appendix}%
\usepackage{xcolor}%
\usepackage{textcomp}%
\usepackage{manyfoot}%
\usepackage{booktabs}%
\usepackage{algorithm}%
\usepackage{algorithmicx}%
\usepackage{algpseudocode}%
\usepackage{listings}%
\usepackage{placeins}
\usepackage{lineno}
\usepackage{setspace}

\newtheorem{theorem}{Theorem}
\newtheorem{remark}{Remark}%
\newtheorem{lemma}{Lemma}%

\def\tsc#1{\csdef{#1}{\textsc{\lowercase{#1}}\xspace}}
\tsc{WGM}
\tsc{QE}

\begin{document}
\let\WriteBookmarks\relax
\def\floatpagepagefraction{1}
\def\textpagefraction{.001}

\shorttitle{}    

\shortauthors{}  

\title [mode = title]{Solution to a  McKean-Vlasov  FBS$\Delta$E and applications to optimal control with law-delay}  

\tnotemark[1] 

\tnotetext[1]{} 

%

\author[concordia]{Duocheng Wang}
\ead{duocheng.wang@concordia.ca}

\address[concordia]{
Department of Mathematics and Statistics,
Concordia University,
Montreal, QC H3G 1M8, Canada
}


\begin{abstract}
In this paper, motivated by discrete-time McKean-Vlasov optimal control problems with law-delay, a class of forward-backward stochastic difference equations is investigated. The main difficulties come from the absence of It$\hat{\rm o}$'s formula in the discrete-time setting, as well as the fully coupled structure in which the coefficients depend on both law-delayed and law-anticipated terms. Under suitable monotonicity conditions, the existence and uniqueness of solutions are established. Furthermore, a new monotonicity condition is introduced to address the associated optimal control problem. As an application, the unique optimal control for linear quadratic systems with law-delay is derived.
\end{abstract}


\begin{highlights}
\item A class of fully coupled McKean–Vlasov forward-backward stochastic difference equations with delayed and anticipated law dependence is investigated, going beyond existing discrete-time models involving only current-law interactions. 
\item The well-posedness of these systems is established under suitable Lipschitz and monotonicity conditions by adapting the continuation method to the discrete-time setting, where the absence of It\^{o}'s formula leads to substantially different derivations from the continuous-time case.
\item A monotonicity condition naturally motivated by discrete-time McKean–Vlasov optimal control problems with law-delay is introduced, and an explicit characterization of the unique optimal control for linear quadratic systems is obtained.
\end{highlights}

\begin{keywords}
forward-backward stochastic difference equations \sep McKean-Vlasov equations \sep continuation method\sep optimal control with law-delay
\end{keywords}

\maketitle
\doublespacing

\section{Introduction}
Forward-backward stochastic differential equations (FBSDEs) have become a fundamental framework in stochastic analysis due to their broad applications in stochastic control, mathematical finance, and related fields since the 1990s. The well-posedness of fully coupled FBSDEs has been extensively studied.  \cite{ma1994solving} proposed the celebrated four-step scheme for solving Markovian FBSDEs. Subsequently, \cite{hu1995solution} and \cite{peng1999fully} established the existence and uniqueness of fully coupled FBSDEs by means of the continuation method.  The theory was further extended by \cite{ma2015well}, who studied the well-posedness of FBSDEs in a general non-Markovian framework. 

McKean–Vlasov stochastic control has become an active research topic in stochastic analysis and control theory. Such distribution-dependent stochastic systems, also known as mean-field stochastic systems, arise naturally in large population systems, interacting particle systems, and mean-field games. The stochastic maximum principle for McKean–Vlasov control systems has been extensively studied; see, for example, \cite{andersson2011maximum,buckdahn2016stochastic,nie2022extended,dong2022maximum,buckdahn2026global} and the references therein. Delay effects have also been incorporated into continuous-time McKean–Vlasov systems to capture memory and hereditary phenomena. In this direction, stochastic maximum principles have been established for systems with state delays or distribution delays; see, e.g., \cite{chen2019maximum,zhang2021stochastic,guo2024stochastic}. It is worth noting that the distribution delays considered in these works do not involve the control variable.

Discrete-time stochastic control has also attracted considerable attention because many practical decision processes are implemented in discrete time. Unlike the continuous-time setting, the absence of Itô's formula makes the analysis substantially different. \cite{lin2014maximum} first established a stochastic maximum principle for non-linear discrete-time optimal control problems. Building upon their pioneering work, numerous studies have further developed the stochastic maximum principle in the discrete-time setting; see, for example, \cite{wu2022maximum,dong2023maximum,ji2024infinite,niu2026stochastic}. However, the existing literature mainly focuses on systems without distribution dependence or involving only current-law interactions. To the best of our knowledge, neither the solvability of fully coupled McKean–Vlasov forward-backward stochastic difference equations (FBS$\Delta$Es) with delayed and anticipated law dependence nor discrete-time optimal control problems with law-delay have been investigated.

To address these problems, we study a class of fully coupled McKean–Vlasov FBS$\Delta$Es motivated by discrete-time mean-field stochastic control with law delay, where both delayed and anticipated law dependence are incorporated into the system dynamics. Compared with the existing solvability results, the analysis is substantially more challenging due to the simultaneous presence of full coupling and delayed and anticipated law dependence. Moreover, the absence of It\^{o}'s formula in the discrete-time setting makes the derivation significantly different from the continuous-time case. Under suitable Lipschitz and monotonicity conditions, we establish the existence and uniqueness of adapted solutions. We further introduce a new monotonicity condition, which is naturally motivated by McKean-Vlasov optimal control problems with law-delay. As an application of the solvability result, we investigate a class of discrete-time linear quadratic (LQ) McKean–Vlasov optimal control problems with law delay and derive an explicit representation of the unique optimal control.

\section{The forward-backward stochastic difference equations
}
Let $(\Omega,\mathcal{F},\mathbb{P})$ be a probability space.  Let $W$ be a $\mathbb{R}^d$-valued square integrable martingale
process with independent increments and the predictable representation property, such that $\mathbb{E}\left[\Delta W_k(\Delta W_l)^\tau\right]= {\delta_{kl}}I_{d\times d}$, where $\Delta W_k=W_{k+1}-W_k$ and $\delta$ is the Kronecker symbol. Let $\mathcal{F}_0\subset \mathcal{F}$ be a sub $\sigma$-algebra independent of $W$, and  {$\mathbb{F}=(\mathcal{F}_k)_{0\le k\le N}$} be the filtration defined by $\mathcal{F}_k=\mathcal{F}_0\lor\sigma(\Delta W_0,...,\Delta W_{k-1})$. We define the following notations:

$$
\begin{aligned}
C^n[a,b]:= & \left\{\varphi_k:[a,b] \rightarrow \mathbb{R}^n \text { is }\mathcal{F}_0 \text {-measurable and } \mathbb{E}\sum_{k\in[a,b]}|\varphi_k|^2<+\infty\right\} \\
L^2\left(\mathcal{F}_N ; \mathbb{R}^m\right):= & \left\{\varphi \text { is } \mathbb{R}^m \text {-valued } \mathcal{F}_{N} \text {-measurable random variable s.t. } \mathbb{E}|\varphi|^2<+\infty\right\} \\
L_{\mathcal{F}}^2\left([a, b] ; \mathbb{R}^n\right):= & \left\{\varphi_k, a \leq k \leq b, \text { is an } \mathcal{G}_k \text {-adapted stochastic process s.t. } \mathbb{E} \sum_{k\in[a,b]}|\varphi_k|^2<+\infty\right\}
\end{aligned}
$$
with $\mathcal{G}_k=\mathcal{F}_{(k\lor 0)\wedge N}$.  {Here $\tau\in(0,N)$ is an integer-valued constant time delay parameter}.

Denote $\mathcal{T}=\{0,1,...,N-1\}$, $\Delta\psi_k=\psi_{k+1}-\psi_k$ for all adapted processes $\psi$. Let $\mathcal{L}(X)$ be the probability distribution of $X$. We consider the following McKean–Vlasov FBS$\Delta$E with delayed and anticipated terms:
\begin{equation}\label{fbsde}
\begin{cases}\Delta x_k=b\left(k, u_k,\alpha(u_k),\alpha(u_{k-\tau}),  \beta(u_k)\right)\\
\qquad\qquad+\sigma\left(k,u_k,\alpha(u_k),\alpha(u_{k-\tau}),\beta(u_k)\right)\Delta W_k, & k\in\mathcal{T}, \\ -\Delta y_k=f\left(k+1,u_{k+1},\alpha(u_{k+1}),\alpha(u_{k+\tau+1})\right) -z_k\Delta W_k, & k\in\mathcal{T}, \\ x_k=\varphi_k, &  {k \in[-\tau, 0]\cup[N+1,N+\tau]}, \\ y_N=\Phi\left(x_N\right), \\
y_k=\xi_k, &  {k \in[-\tau,-1]\cup[N+1, N+\tau]}, \\ z_k=\eta_k, &  {k \in[-\tau,-1]\cup[N, N+\tau]},\end{cases}
\end{equation}
with
\begin{align}\label{notations}
u_k=\left(\begin{array}{l}
x_k \\
y_k \\
z_k
\end{array}\right), 
\quad
\alpha(u_k)=\left(\begin{array}{l}
\mathcal{L}(x_k) \\
\mathcal{L}(y_k)\\
\mathcal{L}(z_k)
\end{array}\right),
\quad
\beta(u_k)=\left(\begin{array}{l}
\mathcal{L}(y_{k+\tau})\\
 \mathcal{L}(z_{k+\tau})
\end{array}\right).
\end{align}
Here $\varphi_k\in C^n[-\tau,0]\cup[N+1,N+\tau]$, $\xi_k\in L_{\mathcal{F}}^2\left([-\tau,-1]\cup[N+1, N+\tau] ; \mathbb{R}^m\right)$, $\eta_k\in L_{\mathcal{F}}^2\left([-\tau,-1]\cup[N, N+\tau] ; \mathbb{R}^{m\times d}\right)$. $\Phi:\Omega\times \mathbb{R}^n\to\mathbb{R}^m$ and
\begin{equation*}
\begin{aligned}
& b: \Omega \times\mathcal{T} \times \mathbb{R}^{n+m+m\times d} \times \left(\mathcal{P}_2(\mathbb{R}^n)\times\mathcal{P}_2(\mathbb{R}^m)\times\mathcal{P}_2(\mathbb{R}^{m\times d})\right)^2\times\left(\mathcal{P}_2(\mathbb{R}^m)\times\mathcal{P}_2(\mathbb{R}^{m\times d})\right)\rightarrow \mathbb{R}^n, \\
& \sigma: \Omega \times\mathcal{T} \times \mathbb{R}^{n+m+m\times d} \times \left(\mathcal{P}_2(\mathbb{R}^n)\times\mathcal{P}_2(\mathbb{R}^m)\times\mathcal{P}_2(\mathbb{R}^{m\times d})\right)^2\times\left(\mathcal{P}_2(\mathbb{R}^m)\times\mathcal{P}_2(\mathbb{R}^{m\times d})\right)\rightarrow \mathbb{R}^{n\times d}, \\
& f: \Omega \times {[1,N]} \times \mathbb{R}^{n+m+m\times d} \times \left(\mathcal{P}_2(\mathbb{R}^n)\times\mathcal{P}_2(\mathbb{R}^m)\times\mathcal{P}_2(\mathbb{R}^{m\times d})\right)^2\rightarrow \mathbb{R}^m.
\end{aligned}
\end{equation*}
For simplicity, we set $b(N,\cdots)=\sigma(N,\cdots)=f(0,\cdots)=0$. Let $G$ be a given $m\times n$ full-rank matrix. Denote
\begin{align}\label{notationA}
\quad \mathcal{A}(k, u)=
\left(\begin{array}{c}
-G^\tau f(k,u) \\
G b\left(k, u\right) \\
G \sigma\left(k, u\right)
\end{array}\right)
:=
\left(\begin{array}{c}
-G^\tau f(k,u_k,\alpha(u_k),\alpha(u_{k+\tau})) \\
G b\left(k, u_k,\alpha(u_k),\alpha(u_{k-\tau}),  \beta(u_k)\right) \\
G \sigma\left(k, u_k,\alpha(u_k),\alpha(u_{k-\tau}),  \beta(u_k)\right)
\end{array}\right).
\end{align}
Throughout this paper, all processes $u=(x,y,z)$ and
$u'=(x',y',z')$ under consideration are assumed to
satisfy the same prescribed initial and exterior boundary values for the corresponding FBS$\Delta$E. Then we state the following assumptions:
\begin{enumerate}
	\item[(H1)] There exists a constant $L>0$, such that for all processes $u=(x,y,z)$ and $u'=(x',y',z')$,
     {
    \begin{align*}
    \sum_{k=0}^N&\left(\|b\left(k, u'\right)-b\left(k, u\right)\|^2+\|\sigma\left(k, u'\right)-\sigma\left(k, u\right)\|^2+\|f(k,u')-f(k,u)\|^2\right)\\   
    \le &L\sum_{k=0}^N\left(\|u'_k-u_k\|^2+\mathbb{E}\|u'_k-u_k\|^2\right).
    \end{align*}
    }
Moreover, for each $u$, $b(\cdot,u),\sigma(\cdot,u),f(\cdot,u)\in L_{\mathcal{F}}^2(0,N)$.
	
\item[(H2)] $\Phi(x)$ is in $L^2\left(\mathcal{F}_N ; \mathbb{R}^m\right)$ and it is uniformly Lipschitz with respect to $x \in \mathbb{R}^n$.

\item[(H3)] There  {exist} constants $\beta_1,\beta_2,\mu\ge0$, such that
\begin{align*}
\mathbb{E}\sum_{k=0}^N\langle\mathcal{A}(k,u')-\mathcal{A}(k,u),u'_k-u_k\rangle\le \mathbb{E}\left[-\sum_{k=0}^{N-1}\beta_1|G\widehat{x}_k|^2-\beta_2\sum_{k=0}^{N-1}\left(|G^\tau\widehat{y}_k|^2+| {G^\tau}\widehat{z}_k|^2\right)\right],
\end{align*}
and $\langle\Phi(x')-\Phi(x), G\widehat{x}\rangle \geq \mu|G \widehat{x}|^2$ for all $x',x\in\mathbb{R}^n$. Moreover, $\beta_1,\mu>0$ (resp. $\beta_2>0$) when $m>n$ (resp. $n>m$) and $\beta_1,\mu>0$ or $\beta_2>0$ as $m=n$.  {Here
\begin{align*}
\langle\mathcal{A}(k,u),u_k\rangle=\langle-G^\tau f(k,u),x_k\rangle+\langle G b(k,u),y_k\rangle+\langle G \sigma\left(k, u\right), z_k\rangle.
\end{align*}}
\end{enumerate}

\begin{remark}
 Most existing works formulate the Lipschitz continuity with respect to probability measures in terms of the Wasserstein distance. For simplicity, we instead adopt the $L^2$-distance  {in (H1)}, since $W_2^2(X,Y)\le \mathbb{E}|X-Y|^2$.
\end{remark}
Then we show the existence and uniqueness of the solution to FBS$\Delta$E (\ref{fbsde}).

\begin{theorem}\label{mainthe}
Assume that (H1)-(H3) hold. Then there exists a unique solution $u=(x,y,z)\in L_{\mathcal{F}}^2\left(-\tau,N ; \mathbb{R}^n\right)\times L_{\mathcal{F}}^2\left(0, N+\tau ; \mathbb{R}^m\right)\times L_{\mathcal{F}}^2\left(0, N+\tau ; \mathbb{R}^{m\times d}\right)$ to the FBS$\Delta$E (\ref{fbsde}).
\end{theorem}

\begin{proof}
We first deal with the case  $m>n, \beta_1,\mu>0$ by introducing the following FBS$\Delta$E with $\varepsilon\in[0,1]$, which becomes FBS$\Delta$E (\ref{fbsde}) when $\varepsilon=1$.
\begin{align}\label{fbsde1}
\begin{cases}\Delta x^\varepsilon_k=\varepsilon b\left(k, u^\varepsilon_k,\alpha(u^\varepsilon_k),\alpha(u^\varepsilon_{k-\tau}),  \beta(u^\varepsilon_k)\right)+\phi_k\\
\qquad\qquad+\left[\varepsilon\sigma\left(k,u^\varepsilon_k,\alpha(u^\varepsilon_k),\alpha(u^\varepsilon_{k-\tau}),\beta(u^\varepsilon_k)\right)+\psi_k\right]\Delta W_k, & k\in\mathcal{T}, \\ 
-\Delta y^\varepsilon_k=\varepsilon f\left(k+1,u^\varepsilon_{k+1},\alpha(u^\varepsilon_{k+1}),\alpha(u^\varepsilon_{k+\tau+1})\right) \\
\qquad\qquad\qquad+(1-\varepsilon)\beta_1Gx^\varepsilon_{k+1}+\gamma_{k+1}-z^\varepsilon_k\Delta W_k, & k\in\mathcal{T}, \\ x^\varepsilon_k=\varphi_k, &  {k \in[-\tau, 0]\cup[N+1,N+\tau]}, \\ y^\varepsilon_N=\varepsilon\Phi\left(x^\varepsilon_N\right)+(1-\varepsilon)\mu Gx_N^\varepsilon+\theta, \\
y^\varepsilon_k=\xi_k, &  {k \in[-\tau,-1]\cup[N+1, N+\tau]}, \\ z^\varepsilon_k=\eta_k, &  {k \in[-\tau,-1]\cup[N, N+\tau]}.\end{cases}
\end{align}
Here $\phi,\psi,\gamma\in L^2_{\mathcal{F}}(0,N)$ with suitable dimensions and $\theta\in L^2(\mathcal{F}_N,\mathbb{R}^m)$.  {For the case $\varepsilon=0$, it is clear that the forward equation has a unique solution, and the solvability of the backward equation follows from Lemma 2.6 of \cite{ji2024infinite}.
Our aim is to prove that there exists  a unique solution when $\varepsilon=1$.} To show the  existence and uniqueness for the FBS$\Delta$E (\ref{fbsde1}) with $\varepsilon\in(0,1]$, we need the following Lemma.
\begin{lemma}\label{existence m>n}
Let $m>n$ and (H1)-(H3) hold. If for some $\varepsilon_0\in[0,1)$, there exists a unique solution $(x^{\varepsilon_0},y^{\varepsilon_0},z^{\varepsilon_0})\in L_{\mathcal{F}}^2\left(-\tau,N ; \mathbb{R}^n\right)\times L_{\mathcal{F}}^2\left(0, N+\tau ; \mathbb{R}^m\right)\times L_{\mathcal{F}}^2\left(0, N+\tau ; \mathbb{R}^{m\times d}\right)$ to the FBS$\Delta$Es (\ref{fbsde1})  {for any $\phi,\psi,\gamma,\theta$}. Then there exists a unique solution $(x^{\varepsilon_0+\delta},y^{\varepsilon_0+\delta},z^{\varepsilon_0+\delta})$ to (\ref{fbsde1}) for $\varepsilon=\varepsilon_0+\delta$  {for any $\phi,\psi,\gamma,\theta$}, where $\delta\in[0,\delta_0]$ for  some constant $\delta_0>0$ independent  {of} $\varepsilon_0$.
\end{lemma}
\begin{proof}
Let $u^i=(x^i,y^i,z^i), i=1,2$, be given processes in $L_{\mathcal{F}}^2\left(-\tau,N ; \mathbb{R}^n\right)\times L_{\mathcal{F}}^2\left(0, N+\tau ; \mathbb{R}^m\right)\times L_{\mathcal{F}}^2\left(0, N+\tau ; \mathbb{R}^{m\times d}\right)$. Under the assumptions, we know there exists a unique $U^i_k=(X^i_k,Y^i_k,Z^i_k),\, i=1,2$, which is the solution to the following FBS$\Delta$E,
\begin{align}\label{fbsde2}
\begin{cases}\Delta X^i_k=\left[\varepsilon_0 b\left(k,U^i\right)+\delta b(k,u^i)+\phi_{k}\right]
&\\
\qquad\quad+\left[\varepsilon_0\sigma\left(k,U^i\right)+\delta \sigma(k,u^i)+\psi_{k}\right] \Delta W_k, & k\in\mathcal{T}, \\ 
-\Delta Y^i_k=\Big[(1-\varepsilon_0)\beta_1GX^i_{k+1}+\varepsilon_0 f\left(k+1,U^i\right)\\
\qquad\qquad+\delta\left(-\beta_1Gx^i_{k+1}+f(k+1,u^i)\right)+\gamma_{k+1}\Big]-Z^i_k\Delta W_k, & k\in\mathcal{T}, \\ 
X^i_k=\varphi_k, &  {k \in[-\tau, 0]\cup[N+1,N+\tau]}, \\  Y^i_N=\varepsilon_0\Phi\left(X^i_N\right)+(1-\varepsilon_0)G\mu X_N^i+\delta\left(\Phi(x^i_N)-\mu Gx_N^i\right)+\theta,&\\
Y^i_k=\xi_k, &  {k \in[-\tau,-1]\cup[N+1,N+\tau]}, \\ Z^i_k=\eta_k, &  {k \in[-\tau,-1]\cup[N,N+\tau]},\end{cases}
\end{align}
where
\begin{align*}
&\rho(k,a)=\rho\left(k, a_k,\alpha(a_k),\alpha(a_{k-\tau}),  \beta(a_k)\right),\quad \rho=b,\sigma,\\
& {f(k,a)=f\left(k,a_{k},\alpha(a_{k}),\alpha(a_{k+\tau})\right)},
\end{align*}
for $a=U^1,U^2,u^1,u^2$.

Then we aim to show the mapping $I: u\to U$ from $L_{\mathcal{F}}^2\left(-\tau,N ; \mathbb{R}^n\right)\times L_{\mathcal{F}}^2\left(0, N+\tau ; \mathbb{R}^m\right)\times L_{\mathcal{F}}^2\left(0, N+\tau ; \mathbb{R}^{m\times d}\right)$ onto itself is a contraction under the norm
\begin{align*}
\|U\|^2=\|(X,Y,Z)\|^2=\mathbb{E}\left[\sum_{k=-\tau}^N|X_k|^2+\sum_{k=0}^{N+\tau}\left(|Y_k|^2+|Z_k|^2\right)\right].
\end{align*}
Denote $\widehat{U}_k=U^1_k-U^2_k, \,\widehat{u}_k=u^1_k-u^2_k$.  {Unlike in the continuous-time setting, we use the discrete-time product rule $\Delta\langle G\widehat{X}_k,\widehat{Y}_k\rangle
=\langle G\Delta\widehat{X}_k,\widehat{Y}_k\rangle+\langle G\widehat{X}_{k+1},\Delta\widehat{Y}_k\rangle$, which yields}
\begin{align}\label{difference}
\Delta\langle G\widehat{X}_k,\widehat{Y}_k\rangle
=&\varepsilon_0\left[\langle -G^\tau \widehat{f}(k+1,U),\widehat{X}_{k+1}\rangle+\langle G\widehat{b}(k,U),\widehat{Y}_k\rangle+\langle G\widehat{\sigma}(k,U)\Delta W_k,\widehat{Z}_k\Delta W_k\rangle\right]\notag\\
&-(1-\varepsilon_0)\beta_1|G\widehat{X}_{k+1}|^2+\langle M_k,\Delta W_k\rangle\notag\\
&+\delta\left[\langle G\widehat{b}(k,u),\widehat{Y}_k\rangle+\langle G\widehat{X}_{k+1},\beta_1G\widehat{x}_{k+1}-\widehat{f}(k+1,u)\rangle+\langle G\widehat{\sigma}(k,u)\Delta W_k,\widehat{Z}_k\Delta W_k\rangle\right],
\end{align}
where { $(M_k)_{0\le k\le N-1}$ is an $(\mathcal{F}_k)_{0\le k\le N-1}$-adapted process} and $\widehat{\rho}(k,a)=\rho(k,a^1)-\rho(k,a^2)$ for $\rho=b,\sigma,f$ and $a=U,u$.
 {We then handle the last term in (\ref{difference}), which differs significantly from its continuous-time counterpart, by
\begin{align*}
\mathbb{E}\langle G\widehat{\sigma}(k,u)\Delta W_k,\widehat{Z}_k\Delta W_k\rangle&=\mathbb{E}[\mathrm{tr}(\langle G\widehat{\sigma}(k,u)\Delta W_k,\widehat{Z}_k\Delta W_k\rangle)]\\
&=\mathbb{E}[\mathrm{tr}(\langle G\widehat{\sigma}(k,u)\mathbb{E}[\Delta W_k\Delta W_k^\tau|\mathcal{F}_k],\widehat{Z}_k\rangle)]\\
&=\mathbb{E}\langle G\widehat{\sigma}(k,u),\widehat{Z}_k\rangle,
\end{align*}
}
through (H1)-(H3), we have
\begin{align*}
\mathbb{E}\langle G\widehat{X}_N,\widehat{Y}_N\rangle=&\mathbb{E}\sum_{k=0}^{N-1}\Delta\langle G\widehat{X}_k,\widehat{Y}_k\rangle\\
=&\varepsilon_0\mathbb{E}\sum_{k=0}^N\langle\mathcal{A}(k,U^1)-\mathcal{A}(k,U^2),\widehat{U}_k\rangle\\
&-(1-\varepsilon_0)\beta_1\mathbb{E}\sum_{k=1}^{N}|G\widehat{X}_k|^2+\delta\mathbb{E}\sum_{k=1}^N\langle G\widehat{X}_{k},\beta_1G\widehat{x}_{k}-\widehat{f}(k,u)\rangle\\
&+\delta\mathbb{E}\sum_{k=0}^{N-1}\left[\langle G\widehat{b}(k,u),\widehat{Y}_k\rangle+\langle G\widehat{\sigma}(k,u),\widehat{Z}_k\rangle\right]\\
\le& -\beta_1\mathbb{E}\sum_{k=0}^{N-1}|G\widehat{X}_k|^2+\delta C\left[\|\widehat{U}\|^2+\|\widehat{u}\|^2\right],
\end{align*}
 {where the term on second line is estimated by the monotonicity condition (H3)}, and $C$ is a positive constant independent of $\varepsilon_0,\delta$, which may vary from line to line. On the other hand, due to (H2) and (H3), we have
\begin{align*}
\mathbb{E}\langle G\widehat{X}_N,\widehat{Y}_N\rangle\ge& \mu\mathbb{E}|G\widehat{X}_N|^2-\delta C\left[|\widehat{X}_N|^2+|\widehat{x}_N|^2\right].
\end{align*}
So we conclude
\begin{align*}
\mathbb{E}\sum_{k=0}^{N}|\widehat{X}_k|^2\le\delta C\left[\|\widehat{U}\|^2+\|\widehat{u}\|^2\right].
\end{align*}
 {For the backward equation, through the Lipschitz continuity of $f$, we have
\begin{align*}
\mathbb{E}\left(|\widehat{Y}_k|^2+|\widehat{Z}_k|^2\right)\le C\mathbb{E}\left(|\widehat{X}_{k+1}|^2+|\widehat{Y}_{k+1}|^2+|\widehat{Z}_{k+1}|^2+|\widehat{X}_{k+\tau+1}|^2+|\widehat{Y}_{k+\tau+1}|^2+|\widehat{Z}_{k+\tau+1}|^2+\delta\left[\|\widehat{U}\|^2+\|\widehat{u}\|^2\right]\right).
\end{align*}
So, by induction, we have
\begin{align*}
\sup_{k\in[0,N]}\mathbb{E}\left(|\widehat{Y}_k|^2+|\widehat{Z}_k|^2\right)\le C\mathbb{E}\sum_{k=0}^N|\widehat{X}_k|^2+\delta C\left[\|\widehat{U}\|^2+\|\widehat{u}\|^2\right].
\end{align*}}
Combining the above estimates with $\sum_{k=-\tau}^0|\widehat{X}_k|^2=\sum_{k=N+1}^{N+\tau}\left(|\widehat{Y}_k|^2+|\widehat{Z}_k|^2\right)=0$, we derive
\begin{align*}
\|\widehat{U}\|^2\le \delta C\left[\|\widehat{U}\|^2+\|\widehat{u}\|^2\right].
\end{align*}
Choosing $\delta_0=\frac{1}{3C}$, it follows
\begin{align*}
\|\widehat{U}\|^2\le \frac{1}{2} \|\widehat{u}\|^2.
\end{align*}
Thus, the mapping $I:u\to U$ is a contraction, so that there exists a unique fixed point in $L_{\mathcal{F}}^2\left(-\tau,N ; \mathbb{R}^n\right)\times L_{\mathcal{F}}^2\left(0, N+\tau ; \mathbb{R}^m\right)\times L_{\mathcal{F}}^2\left(0, N+\tau ; \mathbb{R}^{m\times d}\right)$, which is the unique solution to FBS$\Delta$E (\ref{fbsde1}) for $\varepsilon=\varepsilon_0+\delta$.

This completes the proof of Lemma \ref{existence m>n}.
\end{proof}

For the case $m<n$, similar to the case $m>n$, we have the following Lemma.
\begin{lemma}\label{existence m<n}
Let $m<n$ and (H1)-(H3) hold.  If for some $\varepsilon_0\in[0,1)$ and  {for any $\phi,\psi,\gamma,\theta$}, there exists a unique solution $(x^{\varepsilon_0},y^{\varepsilon_0},z^{\varepsilon_0})\in L_{\mathcal{F}}^2\left(-\tau,N ; \mathbb{R}^n\right)\times L_{\mathcal{F}}^2\left(0, N+\tau ; \mathbb{R}^m\right)\times L_{\mathcal{F}}^2\left(0, N+\tau ; \mathbb{R}^{m\times d}\right)$ to the following FBS$\Delta$E
\begin{align}\label{fbsde4}
\begin{cases}\Delta x^\varepsilon_k=-(1-\varepsilon)\beta_2G^\tau y_k^\varepsilon+\varepsilon b\left(k, u^\varepsilon_k,\alpha(u^\varepsilon_k),\alpha(u^\varepsilon_{k-\tau}),  \beta(u^\varepsilon_k)\right)+\phi_k\\
\qquad\quad+\left[-(1-\varepsilon)\beta_2G^\tau z_k^\varepsilon+\varepsilon\sigma\left(k,u^\varepsilon_k,\alpha(u^\varepsilon_k),\alpha(u^\varepsilon_{k-\tau}),\beta(u^\varepsilon_k)\right)+\psi_k\right]\Delta W_k, & k\in\mathcal{T}, \\ 
-\Delta y^\varepsilon_k=\varepsilon f\left(k+1,u^\varepsilon_{k+1},\alpha(u^\varepsilon_{k+1}),\alpha(u^\varepsilon_{k+\tau+1})\right)+\gamma_{k+1}-z^\varepsilon_k\Delta W_k, & k\in\mathcal{T}, \\ x^\varepsilon_k=\varphi_k, & k \in[-\tau, 0]\cup[N+1,N+\tau], \\ y^\varepsilon_N=\varepsilon\Phi\left(x^\varepsilon_N\right)+\theta, \\
y^\varepsilon_k=\xi_k, & k \in[-\tau,-1]\cup[N+1, N+\tau], \\ z^\varepsilon_k=\eta_k, & k \in[-\tau,-1]\cup[N, N+\tau],\end{cases}
\end{align}
 {then there exists a unique solution }$(x^{\varepsilon_0+\delta},y^{\varepsilon_0+\delta},z^{\varepsilon_0+\delta})$ to (\ref{fbsde4})  {for any $\phi,\psi,\gamma,\theta$}, for $\varepsilon=\varepsilon_0+\delta$, where $\delta\in[0,\delta_0]$ for  some  {constant} $\delta_0>0$ independent of $\varepsilon_0$.
\end{lemma}
\begin{proof}
The proof is similar to that of Lemma \ref{existence m>n} and is therefore omitted.
\end{proof}

Now, through induction, we have proved the existence and uniqueness of the solution to FBS$\Delta$E (\ref{fbsde}) when $m>n$ or $m<n$. For the case $m=n,$ the result follows directly from the case $m>n$ (resp. $m<n$) if $\beta_1,\mu>0$ (resp. $\beta_2>0$). 
\end{proof}

\section{Applications to optimal control}
In this section, we consider the LQ McKean-Vlasov optimal control problem with  {law-delay}. { We assume $d=1$ to simplify the notation.} We first consider the following FBS$\Delta$E.
\begin{equation}\label{fbsde6}
\begin{cases}\Delta x_k=b\left(k, x_k,\mathbb{E}x_k,\mathbb{E}x_{k-\tau},u_{1,k}\right)\\
\qquad\qquad+\sigma\left(k, x_k,\mathbb{E}x_k,\mathbb{E}x_{k-\tau},u_{1,k}\right)\Delta W_k, & k\in\mathcal{T}, \\ -\Delta y_k=f\left(k+1,u_{k+1},\mathbb{E}u_{k+1},\mathbb{E}y_{k+\tau+1},\mathbb{E}z_{k+\tau+1}\right) -z_k\Delta W_k, & k\in\mathcal{T}, \\ x_k=\varphi_k, & k \in[-\tau, 0], \\ y_N=\Phi\left(x_N\right), \\
y_k=\xi_k, &  {k \in[-\tau,-1]\cup[N+1, N+\tau]}, \\ z_k=\eta_k, &   {k \in[-\tau,-1]\cup[N, N+\tau]},\end{cases}
\end{equation}
with
\begin{align*}
u_{1,k}=&\left(y_k,\mathbb{E}y_k,\mathbb{E}y_{k-\tau},\mathbb{E}y_{k+\tau},z_k,\mathbb{E}z_k,\mathbb{E}z_{k-\tau},\mathbb{E}z_{k+\tau}\right).
\end{align*}
Again, we use the notation
\begin{align*}
\quad \mathcal{A}(k,u)=\left(\begin{array}{c}
-f\left(k,u_{k},\mathbb{E}u_{k},\mathbb{E}y_{k+\tau},\mathbb{E}z_{k+\tau}\right) \\
b\left(k, x_k,\mathbb{E}x_k,\mathbb{E}x_{k-\tau},u_{1,k}\right) \\
\sigma\left(k, x_k,\mathbb{E}x_k,\mathbb{E}x_{k-\tau},u_{1,k}\right)
\end{array}\right),
\end{align*}
and assume that
\begin{enumerate}
    \item[(H4)]$b(k,x,\bar{x},\bar{x}_{\tau_-},u_1), \sigma(k,x,\bar{x},\bar{x}_{\tau_-},u_1), f(k,u,\bar{u},\bar{y}_{\tau_+},\bar{z}_{\tau_+}),\Phi(x)$ satisfy a linear growth condition and are Lipschitz continuous with respect to $(x,\bar{x},\bar{x}_{\tau_-},u_1,\bar{u},\bar{y}_{\tau_+},\bar{z}_{\tau_+})$ for all $k\in[0,N]$. { Moreover, there exists a constant $L>0$, such that for all adapted processes $u'=(x',y',z')$, $u=(x,y,z)$ and $\widehat{u}=u'-u$,}
    \begin{align*}
      \mathbb{E}\sum_{k=0}^{N-1}&|b(k,x,\bar{x},\bar{x}_{\tau_-},u_{1,k}')-b(k,x,\bar{x},\bar{x}_{\tau_-},u_{1,k})|^2+|\sigma(k,x,\bar{x},\bar{x}_{\tau_-},u_{1,k}')-\sigma(k,x,\bar{x},\bar{x}_{\tau_-},u_{1,k})|^2\\
        &\le L \mathbb{E}\sum_{k=0}^{N-1}|B_1^\tau\widehat{y}_k+B_2^\tau\mathbb{E}\widehat{y}_k+\mathbf{1}_{\{k<N-\tau\}}B_3^\tau\mathbb{E}\widehat{y}_{k+\tau}+D_1^\tau\widehat{z}_k+D_2^\tau\mathbb{E}\widehat{z}_k+\mathbf{1}_{\{k<N-\tau\}}D_3^\tau\mathbb{E}\widehat{z}_{k+\tau}|^2,
    \end{align*}
    for some given  {deterministic matrices} $B_i, D_i, i=1,2,3$.
    \item [(H5)] There exists a constant $\beta>0$, such that
    \begin{align*}
&\mathbb{E}\sum_{k=0}^N\langle\mathcal{A}(k,u')-\mathcal{A}(k,u),u'_k-u_k\rangle\\\le &-\beta\mathbb{E}\sum_{k=0}^{N-1}|B_1^\tau\widehat{y}_k+B_2^\tau\mathbb{E}\widehat{y}_k+\mathbf{1}_{\{k<N-\tau\}}B_3^\tau\mathbb{E}\widehat{y}_{k+\tau}+D_1^\tau\widehat{z}_k+D_2^\tau\mathbb{E}\widehat{z}_k+\mathbf{1}_{\{k<N-\tau\}}D_3^\tau\mathbb{E}\widehat{z}_{k+\tau}|^2,
\end{align*}
and $\mathbb{E}\langle\Phi(x')-\Phi(x), \widehat{x}\rangle \ge0$ for all $x',x\in L(\mathcal{F}_N;\mathbb{R}^n)$.
\end{enumerate}

Then we give the following result.
\begin{theorem}
Let assumptions (H4) and (H5) hold. Then there exists a unique solution to FBS$\Delta$E (\ref{fbsde6}).
\end{theorem}
\begin{proof}
Consider the following FBS$\Delta$E:
\begin{align*}
\begin{cases}\Delta x^\varepsilon_k=-(1-\varepsilon)\beta\left(B_1u^\varepsilon_{2,k}+B_2\mathbb{E}u_{2,k}^\varepsilon+B_3\mathbb{E}u_{2,k-\tau}^\varepsilon\right)+\varepsilon b\left(k, x_k^\varepsilon,\mathbb{E}x_k^\varepsilon,\mathbb{E}x^\varepsilon_{k-\tau},u^\varepsilon_{1,k}\right)+\phi_k\\
\qquad\quad+\Big[-(1-\varepsilon)\beta\left(D_1u^\varepsilon_{2,k}+D_2\mathbb{E}u_{2,k}^\varepsilon+D_3\mathbb{E}u_{2,k-\tau}^\varepsilon\right)
\\
\qquad\qquad\qquad\qquad\qquad+\varepsilon \sigma\left(k, x_k^\varepsilon,\mathbb{E}x_k^\varepsilon,\mathbb{E}x^\varepsilon_{k-\tau},u_{1,k}^\varepsilon\right)+\psi_k\Big]\Delta W_k, & k\in\mathcal{T}, \\ 
-\Delta y^\varepsilon_k=\varepsilon f\left(k+1,u_{k+1}^\varepsilon,\mathbb{E}u_{k+1}^\varepsilon,\mathbb{E}y_{k+\tau+1}^\varepsilon,\mathbb{E}z_{k+\tau+1}^\varepsilon\right)+\gamma_{k+1}-z^\varepsilon_k\Delta W_k, & k\in\mathcal{T}, \\ x^\varepsilon_k=\varphi_k, & k \in[-\tau, 0]\cup[N+1,N+\tau], \\ y^\varepsilon_N=\varepsilon\Phi\left(x^\varepsilon_N\right)+\theta, \\
y^\varepsilon_k=\xi_k, & k \in[-\tau,-1]\cup[N+1, N+\tau], \\ z^\varepsilon_k=\eta_k, & k \in[-\tau,-1]\cup[N, N+\tau],\end{cases}
\end{align*}
with
\begin{align*}
u^\varepsilon_{2,k}=\mathbf{1}_{\{k\ge 0\}}\left[B^\tau_1y^\varepsilon_k+B^\tau_2\mathbb{E}y^\varepsilon_k+\mathbf{1}_{\{k<N-\tau\}}B^\tau_3\mathbb{E}y^\varepsilon_{k+\tau}+D^\tau_1z^\varepsilon_k+D^\tau_2\mathbb{E}z^\varepsilon_k+\mathbf{1}_{\{k<N-\tau\}}D^\tau_3\mathbb{E}z^\varepsilon_{k+\tau}\right].
\end{align*}
Then the method of proof is similar to that of  Theorem 2.6 of \cite{peng1999fully}.  {We first use (H4) and (H5) to show that
\begin{align*}
I_1:=\mathbb{E}\sum_{k=0}^{N-1}|B_1^\tau\widehat{Y}_k+B_2^\tau\mathbb{E}\widehat{Y}_k+\mathbf{1}_{\{k<N-\tau\}}B_3^\tau\mathbb{E}\widehat{Y}_{k+\tau}+D_1^\tau\widehat{Z}_k+D_2^\tau\mathbb{E}\widehat{Z}_k+\mathbf{1}_{\{k<N-\tau\}}D_3^\tau\mathbb{E}\widehat{Z}_{k+\tau}|^2\le \delta C\left[\|\widehat{U}\|^2+\|\widehat{u}\|^2\right].
\end{align*}
Using the standard estimates for the forward and backward equations, we further have
\begin{align*}
&\mathbb{E}\sum_{k=0}^{N}|\widehat{X}_k|^2\le CI_1+\delta C\left[\|\widehat{U}\|^2+\|\widehat{u}\|^2\right],\\
&\mathbb{E}\sum_{k=0}^{N}|\widehat{Y}_k|^2+|\widehat{Z}_k|^2\le C\mathbb{E}\sum_{k=0}^{N}|\widehat{X}_k|^2+\delta C\left[\|\widehat{U}\|^2+\|\widehat{u}\|^2\right].
\end{align*}
Combining the above estimates, it follows that $\|\widehat{U}\|^2\le \delta C\left[\|\widehat{U}\|^2+\|\widehat{u}\|^2\right]$. Therefore, by choosing $\delta>0$ sufficiently small, the mapping is a contraction.
}
\end{proof}

Then we consider the LQ McKean-Vlasov optimal control problem with law-delay:
\begin{align}\label{LQstate}
\left\{\begin{array}{ll}
\Delta x_k =A_1x_k+A_2\mathbb{E}x_k+A_3\mathbb{E}x_{k-\tau}+B_1v_k+B_2\mathbb{E}v_k+B_3\mathbb{E}v_{k-\tau},\\
\qquad\quad+[C_1x_k+C_2\mathbb{E}x_k+C_3\mathbb{E}x_{k-\tau}+D_1v_k+D_2\mathbb{E}v_k+D_3\mathbb{E}v_{k-\tau}]\Delta W_k&\quad k\in\mathcal{T},
\\x_k=\varphi_k, &k\in[-\tau,0],
\end{array}\right.
\end{align}
to minimize the cost functional
\begin{align}\label{LQcost}
    J(v)=\frac{1}{2}\mathbb{E}\left[\sum_{k=0}^{N-1}\left[\langle Q_1x_k,x_k\rangle+\langle Q_2\mathbb{E}x_k,\mathbb{E}x_k\rangle+\langle Rv_k,v_k\rangle\right]+\langle G_1x_N,x_N\rangle+\langle G_2\mathbb{E}x_N,\mathbb{E}x_N\rangle\right].
\end{align}
 {Here $Q_1, Q_2, G_1,G_2\ge0$ are bounded $\mathcal{F}_0$-measurable symmetric matrices and $R>0$ is a deterministic symmetric matrix such that $\lambda_{min}(R)\ge \alpha_1>0$,} and $A_i,B_i,C_i,D_i,i=1,2,3,$ are  {deterministic matrices} with suitable dimensions.  

 {Denote by $\mathbb{U}$} the set of adapted control processes
 {$v=(v_k)_{k\in[-\tau, N-1]}$} taking values in $\mathbb{R}^r$ such that $v_k=\psi_k, k\in[-\tau,-1]$ and $\mathbb{E}\sum_{k=-\tau}^{N-1}|v_k|^2 <+\infty$. Here  {$\psi_k$} is a given $\mathcal{F}_0$-measurable process.
Our aim is to find the optimal control $v^*$ such that
\begin{align*}
  {   J(v^*)=\inf_{v\in \mathbb{U}}J(v)}
.\end{align*}

\begin{theorem}
The control process defined by
\begin{align}\label{v}
v_k=\begin{cases}\psi_k, & k \in[-\tau, -1], 
\\
 {-R^{-1}\left(B_1^\tau y_k+B_2^\tau \mathbb{E}y_k+\mathbf{1}_{\{k<N-\tau\}}B_3^\tau \mathbb{E}y_{k+\tau}+D_1^\tau z_k+D^\tau _2\mathbb{E}z_k+\mathbf{1}_{\{k<N-\tau\}}D^\tau _3\mathbb{E}z_{k+\tau}\right),} & k \in\mathcal{T},
\end{cases}
\end{align}    
is the unique optimal control for LQ problem (\ref{LQstate}) and (\ref{LQcost}), where $(x,y,z)$ is the solution to the following FBS$\Delta$E:
\begin{equation}\label{fbsde7}
\begin{cases}
\Delta x_k =A_1x_k+A_2\mathbb{E}x_k+A_3\mathbb{E}x_{k-\tau}+B_1v_k+B_2\mathbb{E}v_k+B_3\mathbb{E}v_{k-\tau}\\
\qquad\quad+[C_1x_k+C_2\mathbb{E}x_k+C_3\mathbb{E}x_{k-\tau}+D_1v_k+D_2\mathbb{E}v_k+D_3\mathbb{E}v_{k-\tau}]\Delta W_k,&\quad k\in\mathcal{T},
\\
-\Delta y_k={\mathbf{1}_{\{k<N-1\}}\big[A_1^\tau y_{k+1}+ A_2^\tau\mathbb{E}[ y_{k+1}]+C_1^\tau z_{k+1}+ C_2^\tau \mathbb{E}[z_{k+1}]+Q_1x_{k+1}+\mathbb{E}Q_2\mathbb{E}x_{k+1}}\big]\\
\qquad\qquad\qquad\qquad+ \mathbf{1}_{\{k<N-\tau-1\}}\big[A_3^\tau\mathbb{E}[ y_{k+\tau+1}]+ C_3^\tau \mathbb{E}[z_{k+\tau+1}]\big]-z_k\Delta W_k, & k \in\mathcal{T}, \\
x_k=\varphi_k, &k\in[-\tau,0],\\ 
 y_N=G_1x_N+\mathbb{E}G_2\mathbb{E}x_N, \quad y_k=0, &  {k \in[-\tau,-1]\cup[N+1,N+\tau]}, \\ z_k=0, &  {k \in[-\tau,-1]\cup[N, N+\tau]},\end{cases}
\end{equation}
where $v_k$ is given by Eq. (\ref{v}).
\end{theorem}

\begin{proof}
We first show the existence and uniqueness of the solution to FBS$\Delta$E (\ref{fbsde7}). It is clear that FBS$\Delta$E (\ref{fbsde7}) satisfies (H4). 
For (H5),
\begin{align*}
\mathbb{E}\langle\Phi(x')-\Phi(x),\widehat{x}\rangle
=\mathbb{E}\langle G_1\widehat{x},\widehat{x}\rangle+\mathbb{E}\langle G_2\mathbb{E}\widehat{x},\mathbb{E}\widehat{x}\rangle\ge0,
\end{align*}
and
 {
\begin{align*}
&\mathbb{E}\sum_{k=0}^N\langle\mathcal{A}(k,u')-\mathcal{A}(k,u),u'_k-u_k\rangle\\
=&-\sum_{k=0}^{N-1}\mathbb{E}\left[\langle Q_1\widehat{x}_k,\widehat{x}_k\rangle+\langle Q_2\mathbb{E}\widehat{x}_k,\mathbb{E}\widehat{x}_k\rangle\right]+\mathbb{E}\sum_{k=0}^{N-1}\langle \hat{y}_{k}  ,B_1\widehat{v}_k+B_2\mathbb{E}\widehat{v}_k+B_3\mathbb{E}\widehat{v}_{k-\tau} \rangle\\
&+\mathbb{E}\sum_{k=0}^{N-1}\langle \hat{z}_{k}  ,D_1\widehat{v}_k+D_2\mathbb{E}\widehat{v}_k+D_3\mathbb{E}\widehat{v}_{k-\tau} \rangle\\
\le&\mathbb{E}\sum_{k=0}^{N-1} \langle \hat{v}_{k}  ,B^\tau_1\widehat{y}_k+B^\tau_2\mathbb{E}\widehat{y}_k+\mathbf{1}_{\{k<N-\tau\}}B^\tau_3\mathbb{E}\widehat{y}_{k+\tau} \rangle 
+\mathbb{E}\sum_{k=0}^{N-1} \langle \hat{v}_{k}  ,D^\tau_1\widehat{z}_k+D^\tau_2\mathbb{E}\widehat{z}_k+\mathbf{1}_{\{k<N-\tau\}}D^\tau_3\mathbb{E}\widehat{z}_{k+\tau} \rangle \\
=&-\mathbb{E}\sum_{k=0}^{N-1}\langle R^{-1}\hat{u}_{2,k},\hat{u}_{2,k}\rangle,
\end{align*}
}
with
\begin{align*}
\hat{u}_{2,k}=B^\tau_1\widehat{y}_k+B^\tau_2\mathbb{E}\widehat{y}_k+\mathbf{1}_{\{k<N-\tau\}}B^\tau_3\mathbb{E}\widehat{y}_{k+\tau}+D^\tau_1\widehat{z}_k+D^\tau_2\mathbb{E}\widehat{z}_k+\mathbf{1}_{\{k<N-\tau\}}D^\tau_3\mathbb{E}\widehat{z}_{k+\tau}.
\end{align*}
 {Choosing}  {$\beta=\left(\lambda_{max}(R)\right)^{-1}$}, we show (H5) holds.

Now we show the $v$ given by Eq. (\ref{v}) is an optimal control. Let  {$v'$} be any other admissible control and  {$x'$} be the corresponding state process. Similar to Eq. (\ref{difference}), we have
\begin{align*}
&\mathbb{E}\left[\langle G_1x_N,\widehat{x}_N\rangle+\langle G_2\mathbb{E}x_N,\mathbb{E}\widehat{x}_N\rangle\right]
    \\
    =&\mathbb{E}\sum_{k=0}^{N-1}\Delta \langle \widehat{x}_k,y_k\rangle\\
    =&-\mathbb{E}\sum_{k=0}^{N-1}\left[\langle Q_1x_k,\widehat{x}_k\rangle+\langle Q_2\mathbb{E}x_k,\mathbb{E}\widehat{x}_k\rangle+\langle Rv_k,\widehat{v}_k\rangle\right]\\
    \ge&-\frac{1}{2}\mathbb{E}\sum_{k=0}^{N-1}\left[\langle Q_1x'_k,x_k'\rangle-\langle Q_1x_k,x_k\rangle+\langle Q_2\mathbb{E}x'_k,\mathbb{E}x_k'\rangle-\langle Q_2\mathbb{E}x_k,\mathbb{E}x_k\rangle+\langle Rv'_k,v'_k\rangle-\langle Rv_k,v_k\rangle\right].
\end{align*}
Combining with 
\begin{align*}
\mathbb{E}[\langle G_1x'_N,x'_N\rangle+\langle G_2\mathbb{E}x'_N,\mathbb{E}x'_N\rangle-\langle G_1x_N,x_N\rangle-\langle G_2\mathbb{E}x_N,\mathbb{E}x_N\rangle]\ge2\mathbb{E}[\langle G_1x_N,\widehat{x}_N\rangle+\langle G_2\mathbb{E}x_N,\mathbb{E}\widehat{x}_N\rangle],
\end{align*}
we derive
\begin{align*}
J(v')-J(v)\ge0.
\end{align*}
For uniqueness,  assume that  both $v'$ and $v$ are optimal controls, with $x'$ and $x$ being their corresponding state processes, respectively. It is easy to see that $\frac{x'+x}{2}$ is the corresponding state process to $\frac{v'+v}{2}$. { Recall that $\lambda_{min}(R)\ge\alpha_1>0,\, a.s.,$ and assume that there exists $ \alpha_2\ge 0$, such that $ J(v')=J(v)=\alpha_2.$}
So we have
\begin{align*}
2\alpha_2=&J(v')+J(v)\\
           =&\frac{1}{2}\mathbb{E}\sum_{k=0}^{N-1}[\langle Q_1x_k',x_k'\rangle+\langle Q_1 x_k,x_k\rangle+\langle Q_2\mathbb{E}x_k',\mathbb{E}x_k'\rangle+\langle Q_2\mathbb{E}x_k,\mathbb{E}x_k\rangle+\langle Rv_k',v_k'\rangle+\langle Rv_k,v_k\rangle]\\
           &+\frac{1}{2}\mathbb{E}\left[\langle G_1x_N',x_N'\rangle+\langle G_1x_N,x_N\rangle+\langle G_2\mathbb{E}x_N',\mathbb{E}x_N'\rangle+\langle G_2\mathbb{E}x_N,\mathbb{E}x_N\rangle\right]\\
           \ge&\mathbb{E}\sum_{k=0}^{N-1} \left[\left\langle Q_1\frac{x_k'+x_k}{2},\frac{x_k'+x_k}{2}\right\rangle+\left\langle Q_2\frac{\mathbb{E}x_k'+\mathbb{E}x_k}{2},\frac{\mathbb{E}x_k'+\mathbb{E}x_k}{2}\right\rangle+\left\langle R\frac{v_k'+v_k}{2},\frac{v_k'+v_k}{2}\right\rangle\right]\\
           &+\mathbb{E}\left\langle G_1\frac{x_N'+x_N}{2},\frac{x_N'+x_N}{2}\right\rangle+\mathbb{E}\left\langle G_2\frac{\mathbb{E}x_N'+\mathbb{E}x_N}{2},\frac{\mathbb{E}x_N'+\mathbb{E}x_N}{2}\right\rangle+\mathbb{E}\sum_{k=0}^{N-1}\left\langle R\frac{v_k'-v_k}{2},\frac{v_k'-v_k}{2}\right\rangle\\
           =&2J\Big(\frac{v'+v}{2}\Big)+\mathbb{E}\sum_{k=0}^{N-1}\left\langle R\frac{v_k'-v_k}{2},\frac{v_k'-v_k}{2}\right\rangle \\
           \ge&2\alpha_2+\frac{\alpha_1}{4}\mathbb{E}\sum_{k=0}^{N-1}|v_k'-v_k|^2,
\end{align*}
which implies that $\mathbb{E}\sum_{k=0}^{N-1}|v_k'-v_k|^2=0$.
\end{proof}

\section*{Acknowledgements}
The author sincerely thanks the Editor and the anonymous reviewers for their valuable comments and suggestions. The author also thanks Professor Wei Sun for his helpful guidance and discussions.

\bibliographystyle{cas-model2-names}

\bibliography{cas-refs}



\end{document}